\documentclass[12pt,english]{article}
\usepackage{framed}
\usepackage{amsfonts}
\usepackage{latexsym}
\usepackage{lmodern}
\usepackage[T1]{fontenc}
\usepackage[latin9]{inputenc}
\usepackage[letterpaper]{geometry}
\usepackage{color}
\usepackage{babel}
\usepackage{verbatim}
\usepackage{prettyref}
\usepackage{amsmath}
\usepackage{amsthm}
\usepackage{amssymb}
\usepackage{stmaryrd}
\usepackage{multirow}
\usepackage{float}
\usepackage{array}
\usepackage[labelfont=bf]{caption}
\usepackage{textcomp}
\usepackage{algorithmic}
\usepackage{mathtools}
\usepackage{enumitem}

\usepackage{amsmath}
\usepackage{amsfonts}
\usepackage{latexsym}
\usepackage{amssymb}
\usepackage{framed}

\newtheorem{thm}{Theorem}[section]
\newtheorem{theorem}[thm]{Theorem}

\newtheorem{lemma}[thm]{Lemma}

\newtheorem{proposition}[thm]{Proposition}

\newtheorem{corollary}[thm]{Corollary}

\newtheorem{remark}[thm]{Remark}

\newcommand{\beq}{\begin{equation}}
\newcommand{\eeq}{\end{equation}}
\newcommand{\beqa}{\begin{eqnarray}}
\newcommand{\eeqa}{\end{eqnarray}}
\newcommand{\beqas}{\begin{eqnarray*}}
\newcommand{\eeqas}{\end{eqnarray*}}
\newcommand{\bi}{\begin{itemize}}
\newcommand{\ei}{\end{itemize}}

\newcommand{\vgap}{\vspace{.1in}}

\newcommand{\R}{\mathbb{R}}

\newcommand{\cZ}{{\cal Z}}
\newcommand{\lam}{{\lambda}}

\newcommand{\tlam}{{\tilde \lambda}}

\newcommand{\inner}[2]{\langle #1,#2\rangle}

\newcommand{\argmin}{\mathrm{argmin}}

\newcommand{\dom}{\mathrm{dom}\,}

\newcommand{\bConv}[1]{\overline{\mbox{\rm Conv}}\,(\R^{#1})}

\newcommand{\tx}{\tilde x}

\newcommand{\tu}{\tilde u}
\newcommand{\tv}{\tilde v}

\newcommand{\tz}{\tilde z}

\usepackage[unicode=true,pdfusetitle,
 bookmarks=true,bookmarksnumbered=false,bookmarksopen=false,
 breaklinks=false,pdfborder={0 0 1},backref=false,colorlinks=false]
 {hyperref}
\makeatletter
  \theoremstyle{plain}

\usepackage{babel}

\usepackage{babel}

\makeatother

\begin{document}
\global\long\def\c{\mathbb{C}}

\global\long\def\r{\mathbb{R}}

\global\long\def\R{\mathbb{R}}

\global\long\def\n{\mathbb{N}}

\global\long\def\rn{\mathbb{R}^{n}}


\global\long\def\cZ{{\cal Z}}

\global\long\def\limn{\lim\limits _{n\to\infty}}

\global\long\def\tz{\tilde{z}}

\global\long\def\tx{\tilde{x}}

\global\long\def\lam{\lambda}

\global\long\def\tlam{\tilde{\lambda}}

\title{A Doubly Accelerated Inexact Proximal Point Method for Nonconvex Composite Optimization Problems}
\author{
	Jiaming Liang
	\thanks{School of Industrial and Systems
		Engineering, Georgia Institute of
		Technology, Atlanta, GA, 30332-0205.
		(email: {\tt jiaming.liang@gatech.edu} and {\tt renato.monteiro@isye.gatech.edu}). This work
		was partially supported by NSF Grant CMMI-1300221 and ONR Grant N00014-18-1-2077.}\quad 
	Renato D.C. Monteiro \footnotemark[1]
}
\maketitle
\begin{abstract}
This paper describes and establishes the iteration-complexity of a doubly accelerated inexact proximal point (D-AIPP)
method for solving the  nonconvex composite minimization problem whose objective function is
of the form $f+h$ where $f$ is a (possibly nonconvex) differentiable function whose gradient is Lipschitz continuous and
$h$ is a closed convex function with bounded domain. D-AIPP performs two types of iterations, namely, inner and outer ones.
Its outer iterations correspond to the ones of the accelerated inexact proximal point scheme.
Its inner iterations  are the ones performed by an accelerated
composite gradient method for inexactly solving the convex proximal subproblems generated during the outer iterations.
Thus, D-AIPP employs both inner and outer accelerations.
\end{abstract}

\section{Introduction}\label{sec:int}


Our main goal in this paper is to describe and establish the iteration-complexity of a doubly accelerated inexact proximal point (D-AIPP)
method for solving the  nonconvex composite minimization problem
\begin{equation}\label{eq:ProbIntro}
\phi_*:=\min \left\{ \phi(z):=f(z) + h(z)  : z \in \R^n \right \}
\end{equation}
where $h:\R^n \to (-\infty,\infty]$ is a proper lower-semicontinuous convex function whose domain $\dom h$ is bounded,
and $f$ is a real-valued differentiable (possibly nonconvex) function on a closed convex set $ \Omega \supseteq \dom h $, whose gradient is $M$-Lipschitz continuous on $ \Omega $
and, for some $m \in (0,M]$,  satisfies
\begin{equation}\label{ineq:lowercurv}
f(u) \ge \ell_f(u;z) - \frac{m}{2}\|u-z\|^{2} 
\quad\forall z,u\in \Omega.
\end{equation}
For a given tolerance $\hat \rho>0$, the main result of this paper shows that the 
D-AIPP method obtains a pair $(\hat z, \hat v)$ satisfying
\begin{equation}\label{eq:approxOptimCond_Intro}
\hat v \in \nabla f(\hat z) + \partial h(\hat z), \quad \|\hat v\| \le \hat \rho
\end{equation} 
in at most 
\begin{equation}\label{eq:boundIntro}
\mathcal{O}\left(\frac{M^{1/2}m^{3/2}D^2}{\hat \rho^2} 
+\sqrt{\frac{M}{m}} \log \left(\frac{M+m}{m}\right) \right)
\end{equation}
accelerated gradient iterations
where 
 $D$ denotes the diameter of the domain of $h$.

From an outer perspective, D-AIPP is an accelerated inexact proximal point method reminiscent of ones presented
in the papers \cite{guler1992new,monteiro2013accelerated,salzo2012inexact} whose analyses are carried out only in the context of the convex version of \eqref{eq:ProbIntro}.
A basic ingredient of these methods and the D-AIPP method presented here is the ability of inexactly solving
proximal subproblems of the form
\begin{equation}\label{eq:penPbRegIntro}
\min \left\{ f(u) + h(u) + \frac{1}{2\lam}\|u-\tx_{k-1}\|^2 : u \in \R^n \right \}
\end{equation}
where $\tx_{k-1}$ is a prox-center obtained through a standard acceleration scheme. More specifically,
an iteration of the latter scheme (i.e., the outer one) uses an appropriate approximate solution of the above subproblem and
the one of the previous subproblem to obtain the next $\tx_k$ (see equations \ref{eq:txk} and \eqref{def:x+ FISTA}).
The (outer) iteration-complexity of this accelerated scheme in terms of $\lam$ alone can be shown to be ${\cal O}(1/\lam^2)$
so that the larger $\lam$ is chosen the better its complexity becomes.
It is well-known that when $\lam \le 1/M$, a suitable approximate solution for the above subproblem can be obtained
by exactly solving a single linearized composite minimization problem (referred to as a resolvent evaluation of $h$) of the form
\begin{equation}\label{eq:resolvent}
\min \left\{ \inner{c}{u} + h(u) + \frac{1}{2\lam}\|u\|^2 : u \in \R^n \right \}
\end{equation}
where $c \in \R^n$, namely, the one with $c=\nabla f(\tx_{k-1})-\tx_{k-1}/\lam$.
However, like the methods in \cite{carmon2018accelerated,kong2018complexity}, D-AIPP works with the potentially larger stepsize
$\lambda =1/2m$, which guarantees that the objective function of \eqref{eq:penPbRegIntro} is (actually, strongly) convex
in view of \eqref{ineq:lowercurv}.
Subproblem \eqref{eq:penPbRegIntro} with the latter choice of $\lam$ is then approximately solved 
by an accelerated composite gradient (ACG) variant for convex composite programs (see for example \cite{Attouch2016,YHe2,nesterov1983,nesterov2012gradient,tseng2008accmet})
whose iterations, in the context of D-AIPP,  are referred to as the inner ones.
D-AIPP can therefore be viewed as a doubly accelerated method since it employs acceleration schemes both in its outer iterations
and its inner iterations. 
Using the fact that an iteration of the ACG method only requires a single resolvent evaluation of $h$, the main result of this paper
shows that  the overall number of inner iterations and resolvent evaluations of $h$ required for D-AIPP to obtain
an approximate solution of \eqref{eq:ProbIntro} according to \eqref{eq:approxOptimCond_Intro} is bounded by \eqref{eq:boundIntro}.

\vgap

{\it Related works.} 
Under the assumption that domain of $h$ is bounded, \cite{nonconv_lan16} presents an ACG method applied directly to \eqref{eq:ProbIntro}
 which obtains a $\rho$-approximate solution
of \eqref{eq:ProbIntro} in
\[
{\cal O}\left( \frac{ MmD^2}{{ \rho}^2} + \left(\frac{Md_0}{\rho}\right)^{2/3}\right).
\]
The latter method can be thought as an accelerated method which approximately solves a sequence of subproblems \eqref{eq:resolvent}
with $\lambda=1/M$, which in turn can be solved with single resolvent evaluations of $h$.
We refer to this type of accelerated schemes as short-step ones due to the small magnitude of $\lambda$.
Motivated by \cite{nonconv_lan16}, other papers 
have proposed ACG methods for solving \eqref{eq:ProbIntro}
under different assumptions on the functions $g$ and $h$
(see for example \cite{carmon2018accelerated,Paquette2017,LanUniformly,kong2018complexity,Li_Lin2015,CatalystNC}).
With the exception of \cite{carmon2018accelerated,kong2018complexity}, all these ACG methods are all short-step and, as a result, its complexity in terms of 
$M$ is ${\cal O}(M)$. On the other hand, the ACG methods in \cite{carmon2018accelerated,kong2018complexity} are large-step and  their complexities in terms of $M$ only
are both ${\cal O} (\sqrt{M} \log M)$. The latter two ACG methods perform both outer and inner iterations but, in contrast to the
D-AIPP method, they only perform inner acceleration.

Finally, 
inexact proximal point methods and HPE variants of the ones studied in \cite{monteiro2010complexity,Sol-Sv:hy.ext}
for solving convex-concave saddle point problems and monotone variational inequalities,
which inexactly solve a sequence of proximal suproblems by means of an ACG variant, were previously proposed by \cite{YHe2,YheMoneiroNash,OliverMonteiro,MonteiroSvaiterAcceleration,LanADMM}.
The behavior of an accelerated gradient method near saddle points of unconstrained instances of \eqref{eq:ProbIntro} (i.e., with $h=0$) is studied in \cite{Jwright2017}.

{\it Organization of the paper.}  Subsection~\ref{sec:DefNot} contains basic definitions and notations used in the paper.
Section~\ref{sec:frame} is divided into three subsections. The first one introduces the composite nonconvex optimization (CNO) problem  and discusses some notions of approximate solutions. The second subsection presents a general accelerated inexact proximal point (GAIPP) framework for solving the CNO problem and describes
its global convergence rate bounds. The third subsection is devoted to the proof of the bounds stated in the previous subsection.
Section~\ref{sec:method} contains two subsections. The first one reviews the ACG  method and its properties and
 Subsection~\ref{sec:D-AIPP} presents the D-AIPP method and  its iteration-complexity analysis.
Section~\ref{sec:computResults} presents computational results illustrating the efficiency of the D-AIPP method.
Finally, the appendix contains some complementary results and proofs.

\subsection{Basic definitions and notation} \label{sec:DefNot}

This subsection provides some basic definitions and notations used in this paper.

The set of real numbers is denoted by $\mathbb{R}$. The set of non-negative real numbers  and 
the set of positive real numbers are denoted by $\R_+$ and $\R_{++}$, respectively. We let $\R^2_{++}:=\R_{++}\times \R_{++}$.
Let $\R^n$ denote the standard $n$-dimensional Euclidean 
space equipped with  inner product and norm denoted by $\left\langle \cdot,\cdot\right\rangle $
and $\|\cdot\|$, respectively. For $t>0$, define
$\log^+_1(t):= \max\{\log t ,1\}$. 

The diameter of a set $\Omega \subset \R^n$ is defined as $\sup \{\|z-z'\|:z,z'\in \Omega\}$. In addition, if $ \Omega $ is a nonempty closed convex set, the orthogonal projection $ P_{\Omega}: \R^n \rightarrow \R^n $ onto $ \Omega $ is defined as 
\[
P_{\Omega}(z):=\argmin_{z'\in \Omega} \|z'-z\| \quad \forall z\in \R^n.
\]

Let $\psi: \R^n\rightarrow (-\infty,+\infty]$ be given. The effective domain of $\psi$ is denoted by
$\dom \psi:=\{x \in \R^n: \psi (x) <\infty\}$ and $\psi$ is proper if $\dom \psi \ne \emptyset$.
Moreover, a proper function $\psi: \R^n\rightarrow (-\infty,+\infty]$ is $\mu$-strongly convex for some $\mu \ge 0$ if
$$
\psi(\alpha z+(1-\alpha) u)\leq \alpha \psi(z)+(1-\alpha)\psi(u) - \frac{\alpha(1-\alpha) \mu}{2}\|z-u\|^2
$$
for every $z, u \in \dom \psi$ and $\alpha \in [0,1]$.
If $\psi$ is differentiable at $\bar x \in \R^n$, then its affine approximation $\ell_\psi(\cdot;\bar x)$ at $\bar x$ is defined as
\[
\ell_\psi(z;\bar z) :=  \psi(\bar z) + \inner{\nabla \psi(\bar z)}{z-\bar z} \quad \forall  z \in \R^n.
\]
Also, for $\varepsilon \ge 0$,  its \emph{$\varepsilon$-subdifferential} at $z \in \dom \psi$ is denoted by
\begin{equation}\label{eq:epsubdiff}
\partial_\varepsilon \psi (z):=\left\{ v\in\R^n: \psi(u)\geq \psi(z)+\left\langle v,u-z\right\rangle -\varepsilon,\forall u\in\R^n\right\}.
\end{equation}
The subdifferential of $\psi$ at $z \in \dom \psi$, denoted by $\partial \psi (z)$, corresponds to  $\partial_0 \psi(z)$.
The set of all proper lower semi-continuous convex functions $\psi:\R^n\rightarrow (-\infty,+\infty]$  is denoted by $\bConv{n}$.
%

%
%
%
%
%

\section{A general accelerated IPP framework}\label{sec:frame}

This section contains three subsections. The first one states the CNO problem and 
discusses some notions of approximate solutions. The second subsection presents the GAIPP framework for solving the CNO problem and its global convergence rate bounds. The third subsection provides the detailed technical results for proving the bounds stated in the previous subsection.


\subsection{The CNO problem and corresponding approximate solutions}\label{subsec:approxsol}
This subsection describes the CNO problem which will be the main subject of our analysis in
Subsection \ref{sec:D-AIPP}.
It also presents different notions of approximate solutions for the CNO problem and discusses their relationship. 

The CNO problem we are interested in is \eqref{eq:ProbIntro} where the following conditions are assumed to hold:
\begin{itemize}
\item[(A1)]$h \in \bConv{n}$;  
\item[(A2)] $f$ is a differentiable function on a closed convex set $ \Omega \supseteq \dom h$ and there exist scalars  $M\geq m>0$ such that
\eqref{ineq:lowercurv} holds and
$\nabla f$ is $M$-Lipschitz continuous on $\Omega$, i.e.,
\[
\| \nabla f(u) - \nabla f(z) \| \le M \|u-z\| \quad \forall u,z \in \Omega;
\]
\item[(A3)] 
the diameter $D$ of $\dom h$ is finite.
\end{itemize}

We now make a few remarks about the above assumptions.
First, (A1)-(A3) imply that the set $X_*$ of optimal solutions \eqref{eq:ProbIntro} is non-empty and hence that $\phi_* > -\infty$.
Second, the assumption that $\nabla f$ is $M$-Lipschitz continuous on $\Omega$ in (A2) implies that
\[
- \frac{M}2 \|u-z\|^2  \le f(u) - \ell_f(u;z) \le \frac{M}2 \|u-z\|^2 \quad \forall u,z \in \Omega
\]
and hence that condition \eqref{ineq:lowercurv} which also appears in (A2) is redundant when $m \ge M$. 
Third, it is well-known that a necessary condition for $z^*\in\dom  h$ to be a local minimum of \eqref{eq:ProbIntro} is that
$z^*$ be a stationary point of $f+h$, i.e.,  $0 \in \nabla f(z^*)+\partial h(z^*)$.

The latter inclusion motivates the following notion of approximate solution for problem \eqref{eq:ProbIntro}:
for a given tolerance $\hat \rho>0$,  a pair $( \hat{z},\hat{v})$ is called a $\hat\rho$-approximate solution of \eqref{eq:ProbIntro} if it satisfies
\eqref{eq:approxOptimCond_Intro}.
Another notion of approximate solution that naturally arises in our analysis of the general framework of  Subsection~\ref{subsec:GAIPP} is as follows.
For a given  tolerance pair $(\bar \rho, \bar \varepsilon) \in \R^2_{++}$,  a quintuple $(\lambda,z^-, z, w,\varepsilon) \in \R_{++}\times\R^n \times  \R^n \times \R^n \times \R_+$ is called a $(\bar \rho, \bar \varepsilon)$-prox-approximate solution 
of \eqref{eq:ProbIntro} if
\begin{equation} \label{eq:ref4'}
w \in \partial_{ \varepsilon}\left( \phi + \frac{1}{2\lam} \|\cdot - z^- \|^2 \right)  (z),
\quad \left\|\frac{1}\lam (z^--z)  \right\| \le  \bar \rho, \quad \varepsilon \le  \bar \varepsilon.
\end{equation}

Note that the first definition of approximate solution above depends on the composite structure $(f,h)$ of $\phi$ but the second one does not.

The next proposition, whose proof can be found in \cite{kong2018complexity}, shows how an approximate solution as in \eqref{eq:approxOptimCond_Intro}
can be obtained from a prox-approximate solution by performing a composite gradient step.

\begin{proposition}\label{prop:refapproxsol}
Let $h \in \bConv{n}$, $f$ be a differentiable function on a closed convex set $\Omega \supseteq \dom h$, and its gradient $ \nabla f $ is $M$-Lipschitz continuous on $\Omega$.
Let $(\bar \rho,\bar \varepsilon) \in \R^2_{++}$ and a  $(\bar \rho,\bar \varepsilon)$-prox-approximate solution 
$(\lambda,z^-,z, w,\varepsilon)$ be given  
 and define
\begin{align}\label{eq:def_zg} 
z_f &:= \argmin_u \left\{ \ell_f(u; z) + h(u) + \frac{M+\lambda^{-1}}2 \|u-z\|^2  \right \},\\
q_f&:=[M+\lambda^{-1}](z-z_f),\label{eq:def_qg}\\
v_f &:=  q_f+\nabla f(z_f)  - \nabla f(z).\label{eq:def_vg} 
\end{align}
Then, $(z_f,v_f)$ satisfies 
\[
v_f \in \nabla f(z_f) + \partial  h(z_f), \quad \|v_f\| \le 2 \|q_f\| \le 2 \left  [ \bar\rho + \sqrt{2\bar{\varepsilon} (M+\lam^{-1}) } \right ].
\]
\end{proposition}


\subsection{A general accelerated IPP framework}\label{subsec:GAIPP}

This subsection introduces the GAIPP framework for solving \eqref{eq:ProbIntro}
and presents a global convergence rate result  for it whose proof is given in the next subsection.

We start by stating the GAIPP  framework.

%
%


\noindent\rule[0.5ex]{1\columnwidth}{1pt}

GAIPP Framework

\noindent\rule[0.5ex]{1\columnwidth}{1pt}
\begin{itemize}
	\item [0.] Let $ x_0=y_0\in \dom h $, $0 < \theta < \alpha \le 1$, $ 0< \kappa < 1 $ , $ \delta\ge0 $
	be given, and set $ k=0 $ and $A_0=0$;
	\item [1.] compute
\begin{equation}
		a_k =\frac{1+\sqrt{1+4A_k}}{2} \label{eq:ak}
		\end{equation} 
		\begin{equation} \label{eq:Ak}
		A_{k+1}=A_k+a_k
		\end{equation}
		\begin{equation}
		\tilde x_k= \frac{A_k}{A_{k+1}}y_k+\frac{a_k}{A_{k+1}}x_k;  \label{eq:txk}
		\end{equation}
		\item [2.] choose $ \lam_k>0 $ and find a triple $ (y_{k+1}, \tilde v_{k+1}, \tilde\varepsilon_{k+1}) $ satisfying
			\begin{equation}\label{subdiff}
			\tilde v_{k+1}\in \partial_{\tilde \varepsilon_{k+1}} \left( \lam_k \phi (\cdot) + \frac12 \|\cdot-\tilde x_k\|^2-\frac{\alpha}{2}\|\cdot-y_{k+1}\|^2 \right) (y_{k+1}), 
			\end{equation}
			\begin{equation}\label{ineq:frame}
			\frac{1}{\alpha+\delta}
			\| \tv_{k+1} + \delta(y_{k+1}-\tilde x_k) \|^2 +2\tilde\varepsilon_{k+1} \le (\kappa \alpha + \delta) \|y_{k+1}-\tilde x_k\|^2;
			\end{equation}
		\item[3.] compute
\begin{align}\label{def:x+ FISTA}
                  x_{k+1}&:=\frac{-\tv_{k+1}+\alpha y_{k+1} + \delta x_k/a_k - (1-1/a_k)\theta y_k}{\alpha-\theta+(\theta+\delta)/a_k};
		\end{align}
		\item[4.] set $ k $ \ensuremath{\leftarrow} $ k+1 $ and go to step 1.
	\end{itemize}
\rule[0.5ex]{1\columnwidth}{1pt}

We now make several remarks about the GAIPP framework.
First, GAIPP doest not specify how the triple $ (y_{k+1}, \tilde v_{k+1}, \tilde\varepsilon_{k+1}) $ in step 2 is computed and hence should be viewed as a conceptual framework consisting of (possibly many) specific instances. Second, it assumes that the above triple exists which is generally not the case. Third, when the function inside the parenthesis in \eqref{subdiff} is convex, then it is actually possible to show that a triple
$ (y_{k+1}, \tilde v_{k+1}, \tilde\varepsilon_{k+1}) $ satisfying both \eqref{subdiff} and \eqref{ineq:frame}
can be computed by means of an accelerated gradient method applied to the problem of minimizing the aforementioned function
(see the discussion in Subsection \ref{subsec:ACG}).
Fourth, for the purpose of the analysis of this section, we do not assume that the latter function is convex.
Fifth,
it is easy to see that \eqref{eq:ak} implies that
\begin{equation} \label{eq:rel-theta}
A_k+a_k = a_k^2.
\end{equation}
Sixth,  for any given $\bar \rho >0$, it can be shown with the aid of Proposition \ref{prop:refapproxsol} and \eqref{ineq:frame}
 that, if the quantity $\liminf_{k\to\infty} \|y_{k+1}-\tx_k\|/\lam_k =0$,
then
 the pair $(\hat z,\hat v)$ obtained from formulae \eqref{eq:def_zg}-\eqref{eq:def_vg} with the quintuple
$(\lam,z^-,z,w,\varepsilon)=(\lam_k,\tx_k,y_{k+1},\tv_{k+1},\tilde \varepsilon_{k+1})$
eventually satisfies \eqref{eq:approxOptimCond_Intro}.

The following result whose proof will be given in Subsection \ref{sec:proof} describes how fast the quantity $ \|y_{k+1}-\tx_k\|/\lam_k$ approaches zero.

\begin{theorem} \label{prop:main}
Consider a sequence $\{(y_k,\tx_k,\lam_k)\}$ generated by the GAIPP framework.
Then, for every $ k\ge0 $,
\begin{equation}\label{ineq:r_i}
\min_{0\le i \le k-1} \frac{\|y_{i+1}-\tilde x_i\|^2}{\lam_{i}^2}  	\le 
\frac{ \left[ \theta+\delta+c_0 k	+ 2\left( 1-\theta\right)\sum_{i=0}^{k-1}a_i\right] D^2} {(1-\kappa \alpha) \sum_{i=0}^{k-1}A_{i+1} \lambda_i^2 }
\end{equation}
where $\{a_k\}$ and $\{A_k\}$ are as in \eqref{eq:ak} and \eqref{eq:Ak}, $D$ is as in (A3) and
\begin{equation}\label{eq:beta,tauk}
 c_0 := \frac{4(1-\theta)\beta^2}{(1-\tau_0)^2}, 
\quad \beta:= 3+\frac{4(\theta+\delta)}{\alpha-\theta}, \quad \tau_0 := \frac{\sqrt{\kappa\alpha+\delta}}{\sqrt{\alpha+\delta}}.
\end{equation}
\end{theorem}

We now make a remark about the parameter $c_0$ which appears in \eqref{ineq:r_i}.
First,  $c_0$ depends on the constants
$\alpha$, $\theta$, $\kappa$ and $\delta$  which are given as inputs to the GAIPP framework.
Due to step 0 of the GAIPP framework, the first three constants are ${\cal O}(1)$ but $\delta$ is an arbitrary nonngegative parameter,
and hence it is interesting to examine the dependence of $c_0$ on $\delta$ under the assumption that  $\alpha$, $\theta$ and $\kappa$ are ${\cal O}(1)$, and
the quantities  $(1-\kappa)^{-1}$ and $(\alpha-\theta)^{-1}$ are also ${\cal O}(1)$.
%
Indeed, the definitions of $ \beta $ and $ \tau_0 $ in \eqref{eq:beta,tauk} imply that
	\[
	1-\tau_0=\frac{\sqrt{\alpha+\delta}-\sqrt{\kappa\alpha+\delta}}{\sqrt{\alpha+\delta}}\\
	=\frac{(1-\kappa)\alpha}{\sqrt{\alpha+\delta}(\sqrt{\alpha+\delta}+\sqrt{\kappa\alpha+\delta})}\\
	\ge \frac{(1-\kappa)\alpha}{2(\alpha+\delta)},
	\]
	and hence that
	\[
	\frac\beta{1-\tau_0} \le \frac{2(\alpha+\delta)(3\alpha+\theta+4\delta)}{(1-\kappa)\alpha(\alpha-\theta)}.
	\]
	Thus, it follows fromthe above assumption and  the definition of $c_0$ in  \eqref{eq:beta,tauk} that $c_0=\mathcal{O}(1+\delta^4)$.

The next result considers the special case in which the sequence of
stepsizes $\{\lambda_k\}$ is bounded away from zero and derives
an iteration-complexity bound for the GAIPP framework to obtain an index $k$ such that $ \|y_{k}-\tx_{k-1}\|/\lam_k\le \rho$.
It expresses the bound  not only  in terms of $D$, $\rho$ and the lower bound on $\{\lambda_k\}$ but also in terms of $\delta$, 
and then provides a range of values of $\delta$ for which the bound is minimized.



\begin{corollary}\label{coroll:bound}
Consider a sequence $\{(y_k,\tx_k,\lam_k)\}$ generated by the GAIPP framework and
assume that $(1-\kappa \alpha)^{-1} = {\cal O}(1)$ where $\kappa$ and $\alpha$ are as in step 0 of GAIPP
and that
there exists  $\underline{\lam}>0$ such that
$ \lam_k \ge \underline{\lam} $ for every $ k \ge 0$.
	Then, for any given tolerance $\rho>0$,  there exists an index $k$ such that $ \|y_{k}-\tx_{k-1}\|/\lam_k\le \rho$ and
	\begin{equation}\label{bound}
	k = \mathcal{O}\left( \frac{D^2}{\underline \lam^2 \rho^2} + \frac{\delta^2 D}{ \underline \lambda \rho} + \left( \frac{\delta D^2}{\underline \lam^2\rho^2}\right)^{1/3} + 1 \right).
	\end{equation}
	Consequently, if $ \delta = {\cal O} \left( \sqrt{\underline{\lam} \rho/D} + \sqrt{D/(\underline{\lam} \rho)} \right) $, then the above index $k$ satisfies
 \begin{equation} \label{eq:nene}
k={\cal O}(1+ D^2/(\underline{\lambda}^2 \rho^2) ).
\end{equation}
\end{corollary}
\begin{proof}
	Using \eqref{ineq:r_i} and the bounds presented in Lemma \ref{estimates}, we conclude that
		\[
		\min_{0\le i \le k-1} \frac{\|\tilde x_i-y_{i+1}\|^2}{\lam_{i}^2} \le \frac{D^2}{(1-\kappa \alpha)\underline{\lam}^2}\left[\frac{12(\theta+\delta)}{ k^3}+\frac{12c_0}{k^2}+\frac{8(1-\theta)}{k}\right]
		\]
which, together with the fact that $(1-\kappa \alpha)^{-1} = {\cal O}(1)$ and
$c_0=\mathcal{O}(1+\delta^4)$,  immediately implies the first conclusion of the corollary.
The second conclusion follows immediately from the first one.
\end{proof}

Corollary \ref{coroll:bound} shows that there is a wide range of parameters $\delta$ for which the complexity \eqref{bound} is minimized, i.e., equal to \eqref{eq:nene}.


%
%
%

\subsection{Proof of Theorem \ref{prop:main}}\label{sec:proof}

This subsection provides the proof of Theorem \ref{prop:main} which describes how fast the quantity $ \|y_{k+1}-\tx_k\|/\lam_k$ approaches zero.

Before giving the proof of Theorem \ref{prop:main}, we first state and prove a number of technical results.
The first one introduces two functions that play an important role in our analysis
and establish their main properties.

\begin{lemma}\label{basics}
Letting 
\begin{align}
\label{eq:tildephik}
\tilde \phi_k & := \lambda_k \phi + \frac{1}{2}\|\cdot-\tilde x_k\|^2, \\
\label{eq:gammak}
\gamma_k &:= \tilde \phi_k(y_{k+1})+\inner{\tilde v_{k+1}}{\cdot-y_{k+1}}+\frac{\alpha}2\|\cdot-y_{k+1}\|^2 -
\frac{\theta}2 \|\cdot-\tx_k\|^2-\tilde \varepsilon_{k+1}, 
\end{align}
then, for every $k \ge 0$, the following statements hold:
	\begin{itemize}
	\item[(a)]
	$\gamma_k$ is an $(\alpha-\theta)$-strongly convex quadratic function and
$\tilde \phi_k \ge \gamma_k + \theta \|\cdot-\tx_k\|^2/2$;
	\item[(b)] $ x_{k+1}=\argmin \left \{a_k\gamma_k (u)+ (\theta+\delta) \|u-x_k\|^2/2 : u \in \R^n  \right\} $;
	\item[(c)] $ \min\left\lbrace  \gamma_k(u)+(\theta+\delta) \|u-\tilde x_k\|^2/2 : u \in \R^n \right\rbrace
	\ge \lam_k \phi(y_{k+1})+(1-\kappa\alpha)\|y_{k+1}-\tilde x_k\|^2/2 $;
	\item[(d)] $ \gamma_k (u) - \lambda_k \phi(u) \le (1-\theta)\|u-\tx_k\|^2/2 $ for every $u \in \dom h$.
	\end{itemize}
\end{lemma}
\begin{proof}
		(a) This statement follows immediately from the definitions of $\gamma_k$ and $\tilde \phi_k$  in \eqref{eq:gammak} and \eqref{eq:tildephik},
		respectively, the inclusion \eqref{subdiff} and  the definition of  the $\varepsilon$-subdifferential in \eqref{eq:epsubdiff}.
		
		(b) This statement follows immediately from \eqref{eq:Ak}, \eqref{eq:rel-theta} and the definition of
$\gamma_k$ in \eqref{eq:gammak}.

		(c) Using the definitions of $\gamma_k$ and $\tilde \phi_k$ in \eqref{eq:gammak} and \eqref{eq:tildephik},
		respectively, and inequality \eqref{ineq:frame}, we conclude that 
	\[
		\begin{aligned}
		\min &\left\lbrace \gamma_k(u)+\frac{\theta+\delta}{2}\|u-\tilde x_k\|^2 : u \in \R^n \right\rbrace  \\
		&= \min \left\lbrace \tilde \phi_k(y_{k+1})+\inner{\tilde v_{k+1}}{\cdot-y_{k+1}}+\frac{\alpha}2\|\cdot-y_{k+1}\|^2 +
		\frac{\delta}2 \|\cdot-\tx_k\|^2-\tilde \varepsilon_{k+1} \right\rbrace\\
		& = \tilde \phi(y_{k+1})+  \frac{\delta}2 \|y_{k+1}-\tilde x_k\|^2  - \frac{1}{2(\delta+\alpha)}
		\| v + \delta (y_{k+1}-\tilde x_k) \|^2 -\tilde\varepsilon_{k+1} \\
        & = \lam \phi(y_{k+1})+  \frac{1+\delta}2 \|y_{k+1}-\tilde x_k\|^2  - \frac{1}{2(\delta+\alpha)}
        \| v + \delta (y_{k+1}-\tilde x_k) \|^2 -\tilde\varepsilon_{k+1} \\
        & \ge \lam \phi(y_{k+1})+ \frac{1-\kappa \alpha}2 \|y_{k+1}-\tilde x_k\|^2 
		\end{aligned}
	\]
and hence that (c) holds.

(d) Using the inequality  in (a) and the definition of $ \tilde \phi_k $, we conclude that
\[
\gamma_k + \frac{\theta}2 \|\cdot-\tx_k\|^2\le \tilde \phi_k = \lam_k \phi+\frac{1}{2}\|\cdot-\tx_k\|^2,
\]
and hence that (d) holds.
\end{proof}

The next two results provide formulas relating quantities from $k$-th iteration to those from the $(k+1)$-th iteration.

\begin{lemma}\label{prep}
If $\gamma_k$ is defined as in Lemma \ref{basics}, then for every $k \ge 0$ and $u \in \R^n$:
	\[
		\begin{aligned}
	A_k \gamma_k(y_k)+a_k \gamma_k (u)+\frac{\theta+\delta}2 \|u-x_k\|^2\ge & A_{k+1} \left[\lam_k \phi(y_{k+1} )+\frac{1-\kappa \alpha}{2}\|y_{k+1}-\tilde x_k\|^2\right] \\
	& + \frac{a_k(\alpha -\theta)+\theta+\delta}2 \| u-x_{k+1}\|^2.
	\end{aligned}
	\]
\end{lemma}
\begin{proof}
      Using  the convexity of $\gamma_k$ (see Lemma \ref{basics}(a)), relations \eqref{eq:rel-theta} and \eqref{eq:txk} and Lemma \ref{basics}(c),
      we conclude for every $u \in \R^n $ that
	\begin{align*}
	A_k\gamma_k(y_k)& +a_k\gamma_k(u)+\frac{\theta+\delta}2\|u-x_k\|^2 \\
    & \ge  (A_k+a_k)\gamma_k\left( \frac{A_ky_k+a_ku}{A_k+a_k}  \right)
+\frac{(\theta+\delta)(A_k+a_k)^2}{2a_k^2}\left\| \frac{A_ky_k+a_ku}{A_k+a_k}-\frac{A_ky_k+a_kx_k}{A_k+a_k}\right\| ^2\\
	&= A_{k+1} \left \lbrace \gamma_k \left( \frac{A_ky_k+a_ku}{A_k+a_k} \right) +\frac{\theta+\delta}{2} \left\| \frac{A_ky_k+a_ku}{A_k+a_k}-\tx_k\right\| ^2 \right \rbrace \\
	&\ge A_{k+1} \min\left\lbrace  \gamma_k (\tilde u)+\frac{\theta+\delta}{2}\|\tilde u-\tilde x_k\|^2 :  \tilde u \in \R^n\right\rbrace \\
	&\ge A_{k+1} \left[\lam_k \phi(y_{k+1} )+\frac{1-\kappa \alpha}{2}\|y_{k+1}-\tilde x_k\|^2 \right].
	\end{align*}
      Noting that by Lemma \ref{basics}(a) the left hand side of the above inequality is an $[a_k(\alpha-\theta)+\theta+\delta]$-strongly convex quadratic function,
     the conclusion of the lemma  now follows from Lemma \ref{basics}(b) and the above inequality.
\end{proof}

\begin{lemma} \label{summable 2} Let $x_* \in X_*$ be given and define
	\begin{equation}\label{eta}
	\eta_k :=A_k\lam_k(\phi(y_k)-\phi_*)+\frac {\theta+\delta}2 \|x_k-x_*\|^2 \quad \forall k\ge0.
	\end{equation}
Then, for every $k \ge 0$, the following relations hold:
	\begin{align} 
	\frac{1-\kappa \alpha}{2}A_{k+1} \|y_{k+1} - \tx_k\|^2 & \le  \eta_k - \eta_{k+1}+A_k(\gamma_k(y_k)-\lam_k \phi(y_k))+a_k (\gamma_k (x_*)-\lam_k \phi(x_*)) \nonumber  \\
	\label{ineq:sum}
& \le \eta_k-\eta_{k+1}+\left( 1-\theta\right) a_k D^2 +\left( 1-\theta\right) \|x_k-x_*\|^2.
	\end{align}
\end{lemma}
\begin{proof}
        The first inequality in \eqref{ineq:sum} follows immediately from Lemma \ref{prep} with $u=x_*$, the assumption that $0<\theta< \alpha$, $ \delta\ge0 $ and some simple algebraic manipulations.
	Using Lemma \ref{basics}(d), the definition of $ \tx_k $ in \eqref{eq:txk}, the Cauchy-Schwartz inequality, boundedness of $ \dom h $ and \eqref{eq:rel-theta}, we
conclude that
	\[
	\begin{aligned}
	A_k(\gamma_k(y_k)-\lam_k\phi(y_k))&\le \frac {( 1-\theta)A_k}2\|y_k-\tilde x_k\|^2= \frac{( 1-\theta)A_k}2\cdot\frac{a_k^2}{A_{k+1}^2}\|y_k-x_k\|^2\\
	&\le \frac{(1-\theta)A_ka_k^2}{A_{k+1}^2}(\|y_k-x_*\|^2+\|x_k-x_*\|^2)\\
	&\le \frac{(1-\theta)A_ka_k^2}{A_{k+1}^2} D^2+\left( 1-\theta\right)\frac{A_k}{A_{k+1}}\|x_k-x_*\|^2
	\end{aligned}
	\]
	and
	\[
	\begin{aligned}
	a_k(\gamma(x_*)-\lam_k\phi(x_*))&\le \frac {( 1-\theta)a_k}2\|x_*-\tilde x_k\|^2\\
	&= \frac {( 1-\theta)a_k}2 \left \|\frac{A_k}{A_{k+1}}(y_k-x_*)+\frac{a_k}{A_{k+1}}(x_k-x_*) \right \|^2\\
	&\le \left( 1-\theta\right)a_k\left( \frac{A_k^2}{A_{k+1}^2}\|y_k-x_*\|^2+\frac{a_k^2}{A_{k+1}^2}\|x_k-x_*\|^2\right)\\
	&\le \frac{(1-\theta)A_k^2a_k}{A_{k+1}^2} D^2+  \left(1-\theta\right)  \frac{a_k}{A_{k+1}}\|x_k-x_*\|^2.
	\end{aligned}
	\]
	The second inequality in \eqref{ineq:sum} now follows from the first one in  \eqref{ineq:sum}, the above two inequalities and relation \eqref{eq:Ak}.
\end{proof}

In view of \eqref{subdiff} and assumption (A3), the sequence $\{y_k\}$ is bounded.
The next result shows that $\{x_k\}$ is bounded, and hence that
$\{\tx_k\}$ is also bounded in view of \eqref{eq:txk}.

\begin{lemma}
Let $\beta$ and $\tau_0$ be defined as in \eqref{eq:beta,tauk}. Then,  $\tau_0< 1$ and, for every $\bar x \in \dom h$, we have
\begin{equation}\label{ineq:dkd0}
	\|x_k-\bar x\| \le \tau_0^{k}\|x_0-\bar x\| +\frac{\beta}{1-\tau_0}D \quad \forall k \ge 1.
	\end{equation}
where $D$ is as in (A3). As a consequence, $\{x_k\}$ is bounded.
\end{lemma}

\begin{proof}
The conclusion that $\tau_0<1$ follows from the fact that $\alpha>0$, $\delta>0$ and $\kappa<1$ in view of step 0 of the GAIPP framework.
We now claim that for every $\bar x \in \dom h$,
\begin{equation}\label{ineq:dk}
\|x_{k}-\bar x\| \le \tau_0
\|x_{k-1}-\bar x\| +\beta D \quad \forall k \ge 1,
\end{equation}
which can be easily seen to imply \eqref{ineq:dkd0}.
Let $ \sigma=\kappa\alpha+\delta $. To show the claim, first note that \eqref{ineq:frame}, the definitions of $D$ and $\tx_k$ in (A3) and \eqref{eq:txk}, respectively, the fact that $\bar x \in \dom h$ and $\{y_k\} \subset \dom h$,
and relations \eqref{eq:Ak} and  \eqref{eq:rel-theta}, imply that
	\begin{equation}\label{ineq:tvk}
	\begin{aligned}
	\|\tv_{k+1}+\delta(y_{k+1}-\tx_k)\|
	&\le \sqrt{\sigma(\alpha+\delta)}\|y_{k+1}-\tx_k\| \\
&\le  \sqrt{\sigma (\alpha+\delta)}( \|y_{k+1}-\bar x\|+\|\tx_k-\bar x\|)\\
	&\le  \sqrt{\sigma (\alpha+\delta)}\left( D+\frac{A_k}{A_{k+1}}D+\frac{a_k}{A_{k+1}}\|x_{k}-\bar x\| \right) \\
	& \le \sqrt{\sigma (\alpha+\delta)}\left( 2D+\frac{a_k}{A_{k+1}}\|x_{k}-\bar x\| \right).
	\end{aligned}
	\end{equation}
	From relations \eqref{eq:txk}, \eqref{def:x+ FISTA} and \eqref{eq:rel-theta}, we have
	\begin{equation}\label{eq:bar x}
	\begin{aligned}
	[\alpha-&\theta+(\theta+\delta)/a_k](x_{k+1}-\bar x) \\
	=&-\tv_{k+1}-\delta(y_{k+1}-\tx_k)+\delta\left( \frac{1}{a_k} - \frac{a_k}{A_{k+1}}\right)(x_k-\bar x) \\
	&+(\alpha+\delta)(y_{k+1}-\bar x) -\left( \theta-\frac{\theta}{a_k}+\frac{A_k}{A_{k+1}}\delta \right)(y_k-\bar x) \\
	=&-\tv_{k+1}-\delta(y_{k+1}-\tx_k)+(\alpha+\delta)(y_{k+1}-\bar x)-\frac{A_k(\theta+\delta)}{A_{k+1}}(y_k-\bar x),
	\end{aligned}
	\end{equation}
	Using relations \eqref{ineq:tvk} and \eqref{eq:bar x}, the definition of $D$ in (A3),  and the fact that $\bar x \in \dom h$ and $\{y_k\} \subset \dom h$, we conclude that
	\[
	\begin{aligned}
	& \|x_{k+1}-\bar x\| \\
&\le \frac{1}{\alpha-\theta+(\theta+\delta)/a_k}\left[ \|\tv_{k+1}+\delta(y_{k+1}-\tx_k)\| + (\alpha+\delta)\|y_{k+1}-\bar x\| + \frac{A_k(\theta+\delta)}{A_{k+1}}\|y_k-\bar x\| \right] \\
	&\le \frac{1}{\alpha-\theta+(\theta+\delta)/a_k}\left[ \sqrt{\sigma(\alpha+\delta)}\left( 2D + \frac{a_k}{A_{k+1}}\|x_k-\bar x\| \right) + (\alpha+\delta)D + (\theta+\delta)D \right] \\
	&\le \frac{2\sqrt{\sigma(\alpha+\delta)}+\alpha+2\delta+\theta}{\alpha-\theta+(\theta+\delta)/a_k} D + \frac{\sqrt{\sigma(\alpha+\delta)} a_k}{[\alpha-\theta+(\theta+\delta)/a_k]A_{k+1}}\|x_k-\bar x\|\\
	&\le \beta D + \frac{\sqrt{\sigma(\alpha+\delta)} }{(\alpha-\theta)a_k+\theta+\delta}\|x_k-\bar x\|,
	\end{aligned}
	\]
where the last inequality is due to \eqref{eq:rel-theta},  the definition of $\beta$ in \eqref{eq:beta,tauk} and the local definition of $ \sigma$.
The claim \eqref{ineq:dk} now follows from the previous inequality,
 the fact that $a_0=1$ and $\{a_k\}$ is increasing, $ \sigma=\kappa\alpha+\delta $ and the definition of $\tau_0$ in
\eqref{eq:beta,tauk}.
\end{proof}

\noindent


\begin{lemma}\label{prop:key}
For every $ k\ge0 $,
	\[
	\frac{1-\kappa \alpha}{2}\sum_{i=0}^{k-1}A_{i+1}\|\tilde x_i-y_{i+1}\|^2
	\le \left[ \frac{\theta+\delta}{2} +\left(1-\theta\right)  \frac{2\beta^2 k}{(1-\tau_0)^2}
	+ \left( 1-\theta\right)\sum_{i=0}^{k-1}a_i\right] D^2.
	\]
\end{lemma}

\begin{proof}
Using  \eqref{ineq:dkd0} for $ \bar x=x_* \in X_* $, the fact $ \|x_0-x_*\|\le D $, the inequality $(a_1+a_2)^2 \le 2 (a_1^2 + a_2^2)$ and definitions of $ \beta, \tau_0 $ in \eqref{eq:beta,tauk}, we have
	\[
	\begin{aligned}
	\sum_{i=0}^{k-1}\|x_i-x_*\|^2&\le \|x_0-x_*\|^2+\sum_{i=1}^{k-1}\left( \tau_0^{i}\|x_0-x_*\|+\frac{\beta}{1-\tau_0} D\right) ^2\\
	&\le D^2+2D^2\sum_{i=1}^{k-1}\left( \tau_0^{2i}+\left( \frac{\beta}{1-\tau_0}\right) ^2 \right) \\
	&\le  \frac{2}{1-\tau_0^2} D^2+2\left( \frac{\beta}{1-\tau_0} \right)^2D^2(k-1)
	\le \frac{2\beta^2 k}{(1-\tau_0)^2}D^2. 
	\end{aligned}
	\]
	Moreover,  \eqref{eta} and the facts that $x_* \in X_*$ and $A_0=0$ imply that $\eta_k \ge 0$ and $\eta_0=(\theta+\delta)\|x_0-x_*\|^2/2\le (\theta+\delta) D^2/2$.
         Thus, adding \eqref{ineq:sum} from 0 to $ k-1 $ and using the two previous observations, we obtain
	\[
	\begin{aligned}
	\frac{1-\kappa \alpha}{2} \sum_{i=0}^{k-1}A_{i+1}\|\tilde x_i-y_{i+1}\|^2
	&\le
	\eta_0-\eta_k+\left( 1-\theta\right)D^2\sum_{i=0}^{k-1}a_i + (1-\theta)\sum_{i=0}^{k-1}\|x_i-x_*\|^2\\
	&\le \left[ \frac{\theta+\delta}{2}+ \left( 1-\theta\right)\sum_{i=0}^{k-1}a_i + (1-\theta)\frac{2\beta^2 k}{(1-\tau_0)^2}
	\right] D^2
	\end{aligned}
	\]
which clearly yields the conclusion of the lemma.
\end{proof}

Theorem \ref{prop:main} now follows immediately from Lemma \ref{prop:key}.

\noindent
\section{Accelerated gradient methods}\label{sec:method}

This section presents the D-AIPP method, an instance of the GAIPP framework in which the triple $(y_k,\tv_k, \tilde \varepsilon_k)$ satisfying \eqref{subdiff} and \eqref{ineq:frame} are 
obtained by an ACG variant applied to a certain composite strongly convex minimization subproblem.
It contains two subsections.
The first subsection  reviews the ACG variant and its main properties, and derives a useful result related to the complexity of
finding the aforementioned triple.
The second one presents the D-AIPP method for solving \eqref{eq:ProbIntro} and derives its corresponding  iteration-complexity bound.

\subsection{ACG  method for solving composite convex optimization}\label{subsec:ACG}
In this section we recall an accelerated gradient method as well as some of its properties when applied for solving the following optimization problem
\begin{equation}\label{mainprob:nesterov}
\min \{\psi(z):=\psi_s(z)+\psi_n(z) :   z\in \R^n\}
\end{equation}
where the following conditions hold:
\begin{itemize}
	\item [(B1)]$\psi_n:\r^n\rightarrow (-\infty,+\infty]$ is a proper, closed and $\mu$-strongly convex  function with $\mu\geq 0$;
	\item [(B2)]$\psi_s$ is a convex differentiable function on a closed convex set $ \Omega \supseteq \dom \psi_n $, whose gradient is $L$-Lipschitz continuous on $ \Omega $.
\end{itemize}

The accelerated composite gradient variant (\cite{YHe2,nesterov2012gradient,nesterov1983}) for solving \eqref{mainprob:nesterov} is as follows.

\noindent\rule[0.5ex]{1\columnwidth}{1pt}

\textbf{Accelerated Composite Gradient (ACG) Method}

\noindent\rule[0.5ex]{1\columnwidth}{1pt}
\begin{itemize}
	\item[0.]  Let a pair of functions $(\psi_s,\psi_n)$  be as in \eqref{mainprob:nesterov} 
	and initial point $z_{0}\in \R^n $ be given, and set $y_{0}=z_{0}=P_{\Omega}(z_0)$, $B_{0}=0$, $\Gamma_0\equiv0$ and $j=0$; 
	\item[1.]  compute
	\begin{align*}
	B_{j+1}  &=B_{j}+\frac{\mu B_{j}+1+\sqrt{(\mu B_{j}+1)^2+4L(\mu B_{j}+1)B_{j}}}{2L},\\
	\tilde{z}_{j}  &=\frac{B_{j}}{B_{j+1}}z_{j}+\frac{B_{j+1}-B_{j}}{B_{j+1}}y_{j},\quad\Gamma_{j+1}=\frac{B_{j}}{B_{j+1}}\Gamma_j+\frac{B_{j+1}-B_{j}}{B_{j+1}}l_{\psi_s}(\cdot,\tilde z_j),\\
	y_{j+1} &=\argmin_{y}\left\{ \Gamma_{j+1}(y)+\psi_n(y)+\frac{1}{2B_{j+1}}\|y-y_{0}\|^{2}\right\},\\
	z_{j+1} & =\frac{B_{j}}{B_{j+1}}z_{j}+\frac{B_{j+1}-B_{j}}{B_{j+1}}y_{j+1},
	\end{align*}
	\item[2.] compute 
	\begin{align*}
	u_{j+1}&=\frac{y_0-y_{j+1}}{B_{j+1}},\\[2mm]
	\eta_{j+1}&= \psi(z_{j+1})-\Gamma_{j+1}(y_{j+1})- \psi_n(y_{j+1})-\langle u_{j+1},z_{j+1}-y_{j+1}\rangle;   
	\end{align*}
	\item[3.]  set $j\leftarrow j+1$ and go to step 1. 
\end{itemize}
\noindent\rule[0.5ex]{1\columnwidth}{1pt}

Some remarks about the ACG method follow. First, the main core and usually the common way of describing an iteration of the ACG method is as in step~1.
Second, the
extra sequences $\{u_j\}$ and $\{\eta_j\}$ computed  in step~2 will be used to develop a stopping criterion for the ACG method when the latter
is called as a subroutine in the context of
the D-AIPP method stated in Subsection~\ref{sec:D-AIPP}.
Third, the ACG method in which  $\mu=0$  is a special case of a slightly more general one studied by Tseng in \cite{tseng2008accmet} (see Algorithm~3
of \cite{tseng2008accmet}).
The analysis of the general case of the ACG method in which $\mu\geq0$ was studied  in  \cite[Proposition~2.3]{YHe2}. 

The next proposition summarizes the basic properties of the ACG method.

\begin{proposition}\label{prop1:nesterov} Let $\{(B_j,\Gamma_j,z_j,u_j,\eta_j)\}$     be the sequence generated by the ACG method applied to \eqref{mainprob:nesterov}
where $(\psi_s,\psi_n)$ is a given pair of data functions satisfying (B1) and (B2) with $\mu \geq 0$. Then,  the following statements hold
\begin{itemize}
\item[(a)] for every $j\ge1 $, we have $\Gamma_j\leq \psi_s$ and 
\[
\psi(z_j)  \leq\min_{x}\left\{\Gamma_{j}(z)+\psi_n(z)+\frac{1}{2B_{j}}\|z-z_{0}\|^{2}\right\}
\]
\begin{equation}\label{ineq:increasingA_k}
B_{j}\geq\frac{1}{L}\max\left\{\frac{j^{2}}{4},\left(1+\sqrt{\frac{\mu}{4L}}\right)^{2(j-1)}\right\};
\end{equation}
\item[(b)] for every solution $z^*$  of \eqref{mainprob:nesterov}, we have 
\[
\psi(z_{j})-\psi(z^*)\leq\frac{1}{2B_{j}}\|z^*-z_{0}\|^{2} \qquad \forall j\geq 1;
\]
\item[(c)] for every $j\geq 1$, we have
\begin{equation}
u_j\in  \partial_{\eta_{j}}(\psi_s+\psi_n)(z_j) ,\quad\|B_{j}u_{j}+z_{j}-z_{0}\|^{2}+2B_{j}\eta_{j}\le\|z_{j}-z_{0}\|^{2}.\label{ineq:NestHPE}
\end{equation}
\end{itemize}
\end{proposition}


\begin{proof} 
The proofs of (a) and (b) can be found in Proposition 2.3 of  \cite{YHe2} while the proof of (c) can be found in Proposition 8 of \cite{kong2018complexity}.
\end{proof}

The main role of the accelerated gradient variant of this subsection is to find an approximate solution $y_{k+1}$ of
the subproblem \eqref{subdiff} together with a certificate pair $(\tv_{k+1},\tilde \varepsilon_{k+1})$ satisfying
\eqref{subdiff} and \eqref{ineq:frame}. Indeed, 
we can apply the ACG method with $z_0=\tx_k$ to obtain the triple $(y_{k+1},\tv_{k+1},\tilde \varepsilon_{k+1})$ satisfying \eqref{subdiff} and \eqref{ineq:frame}.

 The following result essentially analyzes the iteration-complexity to compute
the aforementioned triple.

\begin{proposition}\label{cor:inner_complexity}
Let positive constants $ \alpha $, $ \delta $ and $ \kappa $ be given and consider the sequence
 $\{(B_j,\Gamma_j,z_j,u_j,\eta_j)\}$   generated by the ACG method applied to \eqref{mainprob:nesterov}
where $(\psi_s,\psi_n)$ is a given pair of data functions satisfying (B1) and (B2) with $\mu \geq 0$.
If $j$ is an index satisfying
\begin{equation}\label{Ajlower}
	j \ge 2\sqrt{\frac{L(\kappa+1)}{\kappa\alpha+(\kappa+1)\delta}}  \, ,
\end{equation}
 then
\[
\frac{1}{\alpha+\delta} \| u_j + \delta (z_j-z_0) \|^2 +2\eta_j\le (\kappa\alpha+\delta) \|z_j-z_0\|^2.
\]
\end{proposition}
\begin{proof}
Let an index $j$  satisfying \eqref{Ajlower} be given and define $\tau:=\kappa\alpha+(\kappa+1)\delta$.
Noting that  \eqref{ineq:increasingA_k} implies that $B_j\geq \left( \kappa+1 \right) /\left[\kappa\alpha+(\kappa+1)\delta\right] $, and using 
the definition of $\tau$, we easily see after some simple algebraic manipulation that
\beq \label{eq:neww}
B_j - \frac{1}{\tau} \ge \frac{1-\delta B_j}\alpha.
\eeq
Using \eqref{ineq:NestHPE}
and the Cauchy-Schwartz inequality, we easily see that
	\begin{equation} \label{eq:sis}
	B_j\|u_j\|^2+2\eta_j \le 2\inner{u_j}{z_0-z_j} \le \frac{1}{\tau}\|u_j\|^2+\tau\|z_j-z_0\|^2
	\end{equation}
	which, together with \eqref{eq:neww}, 
then imply  that
	\begin{equation}\label{ineq:claim}
	\frac{1-\delta B_j}\alpha \|u_j\|^2+2\eta_j \le \tau \|z_j-z_0\|^2.
	\end{equation}
       Also, the first inequality in \eqref{eq:sis} implies that
	\begin{equation}\label{ineq:decouple}
	\begin{aligned}
	\| u_j + \delta (z_j-z_0) \|^2 
	&\le (1-\delta B_j)\|u_j\|^2 -2\delta\eta_j + \delta^2\|z_j-z_0\|^2.
	\end{aligned}
	\end{equation} 
	Now, \eqref{ineq:claim} and \eqref{ineq:decouple} yield
		\[
	\begin{aligned}
	\frac{1}{\alpha+\delta} \| u_j &+ \delta (z_j-z_0) \|^2 +2\eta_j \\
	&\le\frac{1-\delta B_j}{\alpha+\delta}\|u_j\|^2 + \frac{2\alpha}{\alpha+\delta}\eta_j + \frac{\delta^2}{\alpha+\delta}\|z_j-z_0\|^2\\
	&= \frac{\alpha}{\alpha+\delta}\left( \frac{1-\delta B_j}\alpha \|u_j\|^2+2\eta_j \right) + \frac{\delta^2}{\alpha+\delta}\|z_j-z_0\|^2\\
	&\le \left( \frac{\alpha\tau }{\alpha+\delta} + \frac{\delta^2}{\alpha+\delta} \right) \|z_j-z_0\|^2
	= (\kappa\alpha+\delta) \|z_j-z_0\|^2,
	\end{aligned}
	\]
	and hence that the first conclusion of the proposition holds. The second conclusion of the proposition follows immediately from the first one
	together with relations \eqref{ineq:increasingA_k} and \eqref{ineq:NestHPE}.
\end{proof}
\vspace{2mm} The next result, whose proof is given in Lemma 11 of \cite{kong2018complexity}, contains  some useful relations between two distinct
iterates of the sequence $\{(B_j,z_j,u_j,\eta_j)\}$ generated by the ACG algorithm when $\mu>0$, or equivalently,
$\psi_n$ is strongly convex (see condition (B1)).


\begin{proposition}\label{prop_Nestxjxi}Assume that $\mu>0$ in condition (B1) and let $\{(B_j,z_j,u_j,\eta_j)\}$ be the sequence generated by the ACG algorithm. 
If $i$ is an iteration index such that $B_i\geq \max\{8,9/\mu\}$,  then,  for every $j\geq i$, we have
	\begin{equation}\label{eq:lemNestaaa}
	\|z_j-z_0\| \leq  2\|z_i-z_0\|,\quad \|u_j\|\leq \frac{4}{B_j}\|z_i-z_0\|, \quad
	\eta_j\leq \frac{2}{B_j}\|z_i-z_0\|^2.
	\end{equation}
\end{proposition}

\subsection{Doubly ACG method for solving  the CNO problem}\label{sec:D-AIPP}

This subsection describes and analyzes the D-AIPP method to compute approximate solutions of the CNO problem \eqref{eq:ProbIntro}. The main results of this subsection are Theorem~\ref{th:AIPPcomplexity} and Corollary~\ref{cor:AIPPref2}
which analyze the iteration-complexity of the D-AIPP method to obtain  approximate solutions of the CNO problem in the sense 
of  \eqref{eq:ref4'} and \eqref{eq:approxOptimCond_Intro},  respectively.

We  start by stating  the D-AIPP method.

%
\noindent\rule[0.5ex]{1\columnwidth}{1pt}

D-AIPP

\noindent\rule[0.5ex]{1\columnwidth}{1pt}
\begin{itemize}
	\item [0.] Let $ x_0=y_0\in \dom h $,  a pair $(m,M) \in \R^2_{++}$ satisfying (A2),
a stepsize $ 0 <  \lambda \le 1/(2m)$, a tolerance pair
$(\bar \rho,\bar \varepsilon) \in \R^2_{++}$, and parameters $0<\theta < (1-\lam m)/2$ and $ \delta\ge 0$ be given, and
set $ k=0 $, $A_0=0$ and $ \xi=1-\lam m $.
	\item [1.] compute
		\[
		a_k =\frac{1+\sqrt{1+4 A_k}}{2}, \quad 
		A_{k+1}=A_k+a_k, \quad 
		\tilde x_k= \frac{A_k}{A_{k+1}}y_k+\frac{a_k}{A_{k+1}}x_k;
		\]
		and perform at least $ \left\lceil6\sqrt{2\lam M +1} \right\rceil  $ iterations of the ACG method started from $\tx_k $ and with
		\begin{equation}\label{function}
		\psi_s = \psi_s^k := \lam f+\frac{1}{4}\|\cdot-\tilde x_k\|^2,\quad \psi_n = \psi_n^k := \lam h + \frac{1}4 \|\cdot-\tilde x_k\|^2
		\end{equation}
		to obtain a triple $(z,u,\eta)$ satisfying
        \begin{align}\label{incl, ineq}
		& u \in \partial_{\eta} \left( \lam \phi (\cdot) + \frac12 \|\cdot-\tilde x_k\|^2 \right) (z), \\
& \frac{1}{\xi/2+\delta}
		\| u + \delta (z-\tilde x_k) \|^2 +2\eta \le (\xi/4 + \delta) \|z-\tilde x_k\|^2;\label{ineq:stop-inner}
		\end{align}
	\item[2.] if
		\begin{equation}\label{ineq:stop}
		\|z-\tilde x_k\|\le \frac{\lam \bar \rho}{2}  
		\end{equation}
		then go to step 3; otherwise, set $(y_{k+1},\tv_{k+1},\tilde \varepsilon_{k+1})=(z,u,2\eta)$,
		\begin{equation}\label{eq:x+}
 	 	x_{k+1}:=\frac{-\tv_{k+1}+\xi y_{k+1}/2 + \delta x_k/a_k - (1-1/a_k)\theta y_k}{\xi/2-\theta+(\theta+\delta)/a_k};
		\end{equation}
		and $ k $ \ensuremath{\leftarrow} $ k+1 $, and go to step 1;
\item[3.] restart the previous call to the ACG method in step 1 to find an iterate
$(\tz,\tilde u,\tilde \eta)$ satisfying \eqref{incl, ineq}, \eqref{ineq:stop-inner} with $(z,u,\eta)$ replaced by $(\tz,\tilde u,\tilde \eta)$ and the extra condition
\begin{equation} \label{eq:newcri1}
\tilde \eta  \le \lam \bar\varepsilon
\end{equation}
and set $(y_{k+1},\tilde v_{k+1},\tilde \varepsilon_{k+1})=(\tz,\tilde u,2\tilde \eta)$; finally,
output $(\lambda,y^-,y,v,\varepsilon)$ where $$(y^-,y,v,\varepsilon)=(\tx_k, y_{k+1}, \tv_{k+1}/\lambda, \tilde \varepsilon_{k+1}/(2\lambda)).$$
\end{itemize}
\rule[0.5ex]{1\columnwidth}{1pt}

We now make a few remarks about the D-AIPP method. First, 
the function
$\lam \phi + \|\cdot-\tx_k\|^2/2)$ which appears in \eqref{incl, ineq} clearly has lower curvature equal to $\xi$.
Since  the assumption on $\lambda$  implies that $\xi \ge 1/2$ , the latter function is actually  $\xi$-strongly convex.
In view of the last statement of Proposition~\ref{cor:inner_complexity} with $(\psi_s,\psi_n)$ given by \eqref{function}, the ACG method obtains a triple
$(z,u,\eta)$ satisfying  \eqref{incl, ineq}, \eqref{ineq:stop-inner} or, in the case of the last loop, a triple
$(\tz,\tu,\tilde\eta)$ satisfying \eqref{incl, ineq}, \eqref{ineq:stop-inner} and \eqref{eq:newcri1}.
Second,
the accelerated gradient iterations performed in steps 1 and 3 are referred to as the inner iterations of the D-AIPP method
and  the consecutive loops consisting of steps 0, 1 and 2 (or, steps 0, 1, 2 and 3 in the last loop) are referred to as the outer iterations of the D-AIPP method.
Third, Proposition \ref{prop1:nesterov}(c) implies that every iterate generated by the ACG method
satisfies \eqref{incl, ineq} and hence only \eqref{ineq:stop-inner}
plays the role of a termination criterion to obtain a $(z,u,\eta)$  in step 1.
Finally, the last loop supplements steps 0, 1 and 2  with step 3 whose
goal is to obtain a triple $(\tz,\tilde u,\tilde \eta)$ with a possibly smaller $\tilde \eta$ while
 keeping the quantity $\|\tx_k -\tilde z \|$ on the same order of magnitude as $\|\tx_k -z \|$
(see \eqref{eq:lemNestaaa}).

Our main goal in the remaining part of this subsection is to derive the inner-iteration complexity of the D-AIPP method. This will be accomplished by showing that the D-AIPP method is
a special instance of the GAIPP framework of Subsection \ref{subsec:GAIPP}, as advertised. A useful result towards showing this fact is as follows.

\begin{lemma}\label{inclusion}
	Assume that $\psi \in \bConv{n}$ is a $\xi$-strongly convex function and let
	$(y,\eta) \in \R^n \times \R$ be such that $
	0\in \partial_\eta \psi(y)$.
	Then, 
\[
0 \in \partial_{2\eta}\left(  \psi - \frac{\xi} 4 \|\cdot-y\|^2 \right) (y).
\]
\end{lemma}
\begin{proof}
The assumption on $\psi$ implies that it has a unique global minimum $\bar y $
and that
\begin{equation}\label{ineq:u}
\psi(u)\ge \psi(\bar y)+\frac{\xi}{2}\|u-\bar y\|^2
\end{equation}
 for every $u \in \R^n$. Moreover, the inclusion $ 0\in \partial_\eta \psi(y) $ and the definition of $\varepsilon$-subdifferential in \eqref{eq:epsubdiff}
imply that $\psi(u) \ge \psi(y) - \eta$ for every $u \in \R^n$, and hence that
\begin{equation}\label{ineq:bar y}
\psi(\bar y) \ge \psi(y) - \eta.
\end{equation}
	Hence, we conclude that for every $u \in \R^n$,
	\begin{equation}\label{ineq:u y}
	\begin{aligned}
	\psi(u)&
\ge \psi(y)-\eta+\frac{\xi}{2}\|u-\bar y\|^2 = \psi(y)-\left (\eta+\frac{\xi}{2}\|\bar y-y\|^2 \right) +\frac{\xi}{2}(\|y-\bar y\|^2+\|u-\bar y\|^2)\\
	&\ge \psi(y)-\eta'+\frac{\xi}{4}\|u-y\|^2 
	\end{aligned}
	\end{equation}
	where $ \eta':=\eta+(\xi/2)\|\bar y-y\|^2$. Also, inequality \eqref{ineq:u} with $u=y$
and relation \eqref{ineq:bar y} imply that $(\xi/2) \|\bar y-y\|^2 \le \eta$ and hence that $ \eta' \le 2 \eta$. 
The conclusion now follows from inequality \eqref{ineq:u y}, the definition of $\varepsilon$-subdifferential in \eqref{eq:epsubdiff} and the latter conclusion.
\end{proof}

\begin{lemma}\label{iter-cmplx}
The following statements about the D-AIPP method hold: 
\begin{itemize}
\item[(a)]
it is a special implementation of the GAIPP with $\alpha=\xi/2$, and $ \kappa=1/2 $;
\item[(b)]
if the parameter $\delta \ge 0$ in step 0 satisfies
\begin{equation} \label{eq:delta-cond}
\delta = {\cal O} \left( \sqrt{\frac{\lam \bar \rho}{D}} + \sqrt{\frac{D}{\lam \bar \rho}} \right),
\end{equation}
 then its number of outer iterations is bounded by
\begin{equation}\label{bound:outer}
\mathcal{O}\left( \frac{D^2}{\lam^2\bar \rho^2}+1\right);
\end{equation}
\item[(c)] at every outer iteration, the call to the ACG method in step 2 finds a triple
$(z,u,\eta)$ satisfying \eqref{incl, ineq} and \eqref{ineq:stop-inner} in at most
\[
{\cal O}\left( \sqrt{\lam M + 1} \right)
\]
inner iterations;
\item[(d)] at the last outer  iteration, say the $K$-th one,  the triple $(\tz,\tu,\tilde \eta)$ satisfies
\[
\|\tilde x_K-\tz\| \le \lam \bar \rho, \quad \tilde \eta \le \lam \bar \varepsilon
\]
and the extra number of ACG iterations to obtain such triple is bounded by
\[
{\cal O}\left( 
\sqrt{\lam M+1}\log^+_1 \left(  \frac{\bar \rho\sqrt{\lam^2M+\lam}}{\sqrt{\bar \varepsilon}} \right) 
\right).
\]
\end{itemize}

\end{lemma}
\begin{proof}
	(a) Since $\xi=1-\lam m$ in view of step 0 of the D-AIPP, it follows from\eqref{eq:ProbIntro} and \eqref{ineq:lowercurv} that the function  $\psi= \lam \phi + \|\cdot-\tilde x_k\|^2 /2 - \inner{\tv_{k+1}}{\cdot}$
 is $\xi$-strongly convex. Now, using Lemma \ref{inclusion} to this function $\psi$, we easily see that 
\eqref{incl, ineq} and \eqref{ineq:stop-inner} imply that \eqref{subdiff} and \eqref{ineq:frame} hold with $\alpha=\xi/2$, and $ \kappa=1/2 $.
	Statement (a) now  follows by noting that \eqref{eq:x+} is equivalent to \eqref{def:x+ FISTA} with $\alpha=\xi/2$.
	
	(b) Using (a) and Corollary \ref{coroll:bound}, we can easily see that the number of outer iterations needed to satisfy the
termination criterion \eqref{ineq:stop} for the outer iteration is bounded by \eqref{bound:outer}.
	
	(c) First, note that the function $\psi_s$ defined in \eqref{function} satisfies condition (B2) of Subsection~\ref{subsec:ACG}
with  $ L=\lambda M+1/2$,  in view of  \eqref{ineq:lowercurv}.
Hence, it follows from the last conclusion of Proposition~\ref{cor:inner_complexity} that the ACG method  obtains a triple $(z,u,\eta)$ satisfying
\eqref{incl, ineq} and \eqref{ineq:stop-inner} in at most  
\[
\left\lceil 2\sqrt{\frac{6L}{\xi+6\delta}} \right\rceil = \left\lceil 2\sqrt{\frac{3(2\lam M + 1)}{\xi+6\delta}} \right\rceil 
\]
inner iterations.
Since the call to ACG method in step 1 is required to be at least $ \left\lceil6\sqrt{2\lam M+1}
\right\rceil  $ iterations,
we then conclude that the call to the ACG method in step 1 finds a triple
$(z,u,\eta)$ satisfying \eqref{incl, ineq} and \eqref{ineq:stop-inner}in at most
\begin{equation}\label{eq:inner}
	\left\lceil 2\sqrt{2\lam M+1}\max \left\lbrace 3, \sqrt{\frac{3}{\xi+6\delta}} \right\rbrace \right\rceil =\left\lceil 6\sqrt{2\lam M+1}\right\rceil = {\cal O}\left( \sqrt{\lam(M+m)} \right)
\end{equation}
where the first equality is due to the fact that $\xi =1 - \lam m$ and $\lam \ge 1/(2m)$ (see step 0 of the D-AIPP method).

%
	(d)
Consider the triple $(z,u,\eta)$ obtained in step 1 during the last outer iteration of 
the D-AIPP method. Clearly, $(z,u,\eta)$ satisfies \eqref{incl, ineq}, \eqref{ineq:stop-inner} and \eqref{ineq:stop}.
In view of step 1, there exists $i \ge \left\lceil 6\sqrt{2\lambda M+1}\right\rceil$ such that
$(z,u,\eta)$ is the $i$th-iterate of the ACG method started from $z_0=\tx_K$ applied to problem \eqref{mainprob:nesterov} with $\psi_s$ and $\psi_n$ as in \eqref{function}.
Noting that  the functions  $\psi_n$ and $\psi_s$ satisfy conditions (B1) and  (B2)
of Subsection~\ref{subsec:ACG} with  $\mu=1/2$ and $L= \lambda M+1/2$ (see \eqref{ineq:lowercurv}) and using
the above inequality on the index $i$ and relation \eqref{ineq:increasingA_k},
we conclude that $B_i\geq 18=\max\{8,9/\mu\}$, and hence that $i$ satisfies  the assumption of Proposition~\ref{prop_Nestxjxi}. 
It then follows from  \eqref{eq:lemNestaaa}, \eqref{ineq:stop-inner}, \eqref{ineq:stop}  and \eqref{ineq:increasingA_k} that the continuation of the ACG method mentioned in step~3 generates
a  triple  $(\tz,\tu,\tilde \eta)=(z_j,u_j,\eta_j)$ satisfying 
\begin{align*} 
&\|\tz-\tilde x_K\| \le 2 \|z-\tx_K\|\le \lam \bar \rho, \\
 & \tilde\eta\le \frac{2}{B_j}\|z-\tx_K\|^2\le \frac{2L\lam^2\bar\rho^2}{9\left( 1+\sqrt{\frac{\mu}{4L}}\right) ^{2(j-1)}}.
\end{align*}
The last inequality together with the fact that  $\mu=1/2$ and  $L= \lambda M+1/2$,  and $\log(1+t)\geq t/2$
for all $t \in [0,1]$,  can now be easily seen to imply the conclusion of (d).
\end{proof}

We now make some remarks about Lemma \ref{iter-cmplx}.
First, D-AIPP also converges when $\delta$ does not satisfy \eqref{eq:delta-cond} but the corresponding outer-iteration complexity bound will be worse
than the one in Lemma \ref{iter-cmplx} (b). Second, the inner iteration-complexity bound for D-AIPP is independent of $\delta$
due to the fact that the largest term in the maximand in \eqref{eq:inner} is the first one.
The latter term is due to the fact that step 1 of the D-AIPP prescribes that a number of inner iterations on the order of this term
be performed.
Third, we have implemented a more practical (although heuristic) variant of D-AIPP
in which this requirement is not fulfilled and have observed that
its inner iteration-complexity behaves more like the second term, i.e., as ${\cal O}(1/\sqrt {1+\delta})$.
In summary, although the whole notation, formulae and analysis presented so far would become much simpler by simply choosing $\delta=0$,
the above remarks show how important is to make $\delta$ large.
%

We next state the main result of this paper which derives the iteration-complexity of the D-AIPP method to obtain a
$(\bar \rho,\bar \varepsilon)$-prox-approximate solution
of  the CNO problem in the sense of~\eqref{eq:ref4'}. 

\begin{theorem}\label{th:AIPPcomplexity}
If the parameter $\delta \ge 0$ in step 0 satisfies \eqref{eq:delta-cond}, then
 the D-AIPP method terminates with a $(\bar \rho,\bar \varepsilon)$-prox-approximate solution  $(\lam,y^-,y,v,\varepsilon)$ 
 by performing a total number of inner iterations bounded by
\begin{equation}\label{eq:lastboundAIPP}
\mathcal{O}\left\{ \sqrt{\lam M+1}\left[  \frac{D^2}{\lam^2\bar \rho^2}
 + \log^+_1\left(\frac{ \bar\rho \sqrt{\lam^2M+\lam}}{\sqrt{\bar\varepsilon}}\right) \right]
\right\}.
\end{equation}
As a consequence, if $ \lam = \Theta\left( 1/m \right)  $, the above inner-iteration complexity reduces to
\[
\mathcal{O}\left( \frac{M^{1/2}m^{3/2}D^2}{\bar \rho^2} + \sqrt{\frac{M}{m}}\log ^+_1 \left( \frac{\bar \rho \sqrt{M}}{m\sqrt{\bar \varepsilon}}\right) \right).
\]
\end{theorem}
\begin{proof}
It is easy to see that the definition of $(\lambda,y^-,y,v,\varepsilon)$ in step~3, the fact that $(\tz,\tilde u,\tilde \eta)$ satisfies \eqref{incl, ineq} and \eqref{ineq:stop-inner} with $(z,u,\eta)$
replaced by $(\tz,\tilde u,\tilde \eta)$, some elementary properties of the subdifferential,
and part (d) of the previous lemma, imply that the output  $(\lambda,y^-,y,v,\varepsilon)$  satisfies \eqref{eq:ref4'}, and hence is a $(\bar \rho, \bar \varepsilon)$-prox-approximate solution
of \eqref{eq:ProbIntro}. The bound in \eqref{eq:lastboundAIPP} follows from the bounds derived in statements (b)-(d) of Lemma~\ref{iter-cmplx}.
\end{proof}

\vgap
The following result, which follows as an immediate consequence of Theorem \ref{th:AIPPcomplexity}, analyzes the overall number of inner iterations of the D-AIPP method to compute approximate solutions of the CNO problem in the sense of \eqref{eq:approxOptimCond_Intro}.

\begin{corollary}\label{cor:AIPPref2}

Assume that the parameter $\delta \ge 0$ in step 0 satisfies \eqref{eq:delta-cond}.
Also, let a tolerance $\hat\rho>0$ be given and let $(\lambda,y^-,y,v,\varepsilon)$
be the output obtained by the  D-AIPP method with inputs  $\lambda=1/(2m)$ and $(\bar\rho,\bar\varepsilon)$ defined as
\begin{equation}\label{eq:cor_complex1}
\bar\rho:=\frac{\hat \rho}{4} \quad \mbox{and}\quad \bar\varepsilon:= \frac{\hat \rho^2}{32(M+2m)}.
\end{equation}
Then, the following statements hold:
\begin{itemize}
\item[(a)] the D-AIPP method  terminates in at most
\[
\mathcal{O}\left(\frac{M^{1/2}m^{3/2}D^2}{\hat \rho^2}
+\sqrt{\frac{M}{m}} \log^{+}_1\left(\frac{M}{m}\right)\right)
\]
inner iterations;
\item[(b)]
if  $\nabla g$ is $M$-Lipschitz continuous, then the pair  $(\hat z, \hat v)=(z_g,v_g)$ computed  according to  \eqref{eq:def_zg} and \eqref{eq:def_vg}  
is a $\hat{\rho}$-approximate solution of \eqref{eq:ProbIntro}, i.e., \eqref{eq:approxOptimCond_Intro} holds. 
\end{itemize}
\end{corollary}
\begin{proof}
(a) This statement follows immediately from  Theorem~\ref{th:AIPPcomplexity}, \eqref{eq:cor_complex1} and the fact that $m\leq M$ due to (A2).

 (b) First note that Theorem~\ref{th:AIPPcomplexity} implies that the D-AIPP output $(\lambda,z^-,z,v,\varepsilon)$ satisfies criterion \eqref{eq:ref4'} with  $\bar\rho$ and $\bar\varepsilon$  as  in \eqref{eq:cor_complex1}.  Moreover,  \eqref{eq:cor_complex1} also implies that 
$$
\hat\rho= 2 \left[ \bar\rho+ \sqrt{2\bar\varepsilon \left(M+2m\right)} \right].
$$
Hence,  the result   follows from Proposition~\ref{prop:refapproxsol} and the fact that $\lambda=1/2m$.
\end{proof}
\noindent
\noindent
\section{Computational Results}
\label{sec:computResults}

In this section, a few computational results are presented to show the performance of the D-AIPP method. The D-AIPP method is benchmarked against two other nonconvex optimization methods, namely the accelerated 
gradient (AG) method \cite{nonconv_lan16} and the accelerated 
inexact proximal point (AIPP) method recently proposed and analyzed in \cite{kong2018complexity}.

We consider the quadratic programming problem
\raggedbottom
\begin{equation}\label{testQPprob}
\min\left\{ f(z):=-\frac{\alpha_1}{2}\|DBz\|^{2}+\frac{\alpha_2}{2}\|Az-b\|^{2}:z\in\Delta_{n}\right\}
\end{equation}
where $A\in \R^{l\times n}$,  $B\in \R^{n\times n}$,  $D\in \R^{n\times n}$ is a diagonal matrix, $b\in \R^{ l}$, $(\alpha_1,\alpha_2)\in \R^2_{++}$,  and
$\Delta_{n}:=\left\{ z\in\R^n:\sum_{i=1}^{n}z_{i}=1, \;\; z_i\geq0\right\}$.
More specifically, we set the dimensions to be $(l,n)=(20,300)$.  We also generate the entries of $A,B$ and $b$ by sampling from the uniform distribution ${\cal U}[0,1]$ 
and the diagonal entries of $D$ by sampling from the discrete uniform distribution ${\cal U}\{1,1000\}$.  By appropriately choosing the scalars  $\alpha_1$ and $\alpha_2$,
the instance corresponding to a pair of parameters $(M,m)\in\R^2_{++}$
was generated so that $M=\lambda_{\max}(\nabla^{2}f)$ and $-m=\lambda_{\min}(\nabla^{2}f)$ where $\lambda_{\max}(\nabla^{2}f)$ and $\lambda_{\min}(\nabla^{2}f)$ denote
the largest and smallest eigenvalues of the Hessian of $f$ respectively.
We choose the parameters $\lambda$, $\theta$ and $\delta$ so that
$(\lambda, \theta, \theta+\delta)=(0.9/m, 0.49\xi, 0.9(M/m)^{1/7})$. The AIPP method is implemented with parameters $\sigma=0.3$ and $\lambda=0.9/m$. The AG, AIPP and D-AIPP methods are implemented in 
MATLAB 2017b scripts and are run on a MacBook Pro containing a 4-core Intel Core i7 processor and 16 GB of memory. 

All three methods use the centroid of the set
$\Delta_{n}$ as the initial point $z_0$  and are run until a pair $(z,v)$ is generated satisfying the condition
\[
v\in \nabla f(z)+N_{\Delta_n}(z), \qquad  \frac{\|v\|}{\|\nabla f(z_{0})\|+1}\leq \bar \rho\label{eq:comp_term}
\]
for a given tolerance $\bar \rho>0$. Here, $N_{X}(z)$ denotes the normal cone of $X$ at $z$, i.e. $N_X(z)=\{u\in \R^n: \langle u, \tilde z-z\rangle \leq 0,\; \forall \tilde z \in X\}.$ 
The results of the table below is obtained with $\bar \rho=10^{-7}$ and presents results for different choices of the curvature pair $(M,m)$.  Each entry  in the $\bar{f}$-column is the value of the objective function of \eqref{testQPprob}
at the last iterate generated by each method. Since they are approximately the same for all three methods, only one value is reported.
The bold numbers in each table highlight the algorithm that performed the most efficiently in terms of total number of iterations. 


\begin{table}[H]\label{tab:t1}
	\begin{centering}
		\begin{tabular}{|>{\centering}p{1.5cm}>{\centering}p{1.5cm}|>{\centering}m{2cm}|>{\centering}p{1.6cm}>{\centering}p{1.6cm}>{\centering}p{1.6cm}|}
			\hline 
			\multicolumn{2}{|c|}{Size} & \multirow{2}{2cm}{\centering{}$\bar{f}$} & \multicolumn{3}{c|}{Iteration Count}\tabularnewline
			{\small{}$M$} & {\small{}$m$} &  & {\small{}AG} & {\small{}AIPP} &  {\small{}D-AIPP}\tabularnewline
			\hline 
			{\small{}16777216} & {\small{}1048576} & {\small{}-3.83E+04} & {\small{}4429} & {\small{}6711} & \textbf{{\small{}1246}}\tabularnewline
			{\small{}16777216} & {\small{}65536} & {\small{}-4.46E+02} & {\small{}22087} & {\small{}24129} & \textbf{\small{}4920}\tabularnewline
			{\small{}16777216} & {\small{}4096} & {\small{}4.07E+03} & {\small{}26053} & {\small{}5706} & \textbf{\small{}5585}\tabularnewline
			{\small{}16777216} & {\small{}256} & {\small{}4.38E+03} & {\small{}20371} & \textbf{\small{}1625} & {\small{}2883}\tabularnewline
			{\small{}16777216} & {\small{}16} & {\small{}4.40E+03} & {\small{}20761} & \textbf{\small{}2308} & {\small{}3656}\tabularnewline
			\hline 
		\end{tabular}
		\par\end{centering}
	\caption{Numerical results for the AG,  AIPP and D-AIPP methods}
\end{table}
\vspace{-5mm}

From the table, we can conclude that if the curvature ratio $M/m$ is sufficiently small then the D-AIPP method performs fewer iterations than both the AG and the AIPP methods. 
However, for the large curvature ratio cases, D-AIPP is also competite, although not as good as the AIPP method. Nevertheless, preliminary computational results seem to indicate
that, by varying $ \lam $ and $ \delta $ adaptively, a variant of D-AIPP can also efficiently solve instances with large curvature ratios.
We leave this preliminary investigation for a future work, since this situation is not covered by the theory presented in this paper, 

\bibliographystyle{plain}
\bibliography{DAIPP.bib}

\appendix
\section{Technical results}\label{Appendix A}

This appendix presents some relations involving the two sequences of parameters $ \{a_k\} $ and $ \{A_k\} $ used both in
the GAIPP framework and the D-AIPP method.

\begin{lemma}\label{estimates}
	For every $ k\ge 0 $, the sequences $ \{a_k\} $ and $ \{A_k\} $ defined in \eqref{eq:ak} and \eqref{eq:Ak}, respectively, satisfy
	\[
	A_k\ge\frac{k^2}{4}, \quad \sum_{i=0}^{k-1}A_{i+1}\ge\frac{k^3}{12}, \quad \frac{\sum_{i=0}^{k-1}a_i}{\sum_{i=0}^{k-1}A_{i+1}}\le \frac{4}{k}.
	\]
\end{lemma}
\begin{proof}
	The proof of the first inequality can be found in Lemma 7 of \cite{nesterov2012gradient}.
	The second inequality follows directly from the first one.
	Now, using \eqref{eq:Ak} and \eqref{eq:rel-theta}, the Cauchy-Schwartz inequality, and the first inequality of the lemma, we conclude that the last inequality of the lemma
holds as follows:
	\[
	\frac{\sum_{i=0}^{k-1}a_i}{\sum_{i=0}^{k-1}A_{i+1}}=\frac{\sum_{i=0}^{k-1}a_i}{\sum_{i=0}^{k-1}a_i^2}
	\le \frac{\sum_{i=0}^{k-1}a_i}{\frac1k (\sum_{i=0}^{k-1}a_i)^2}
	= \frac{k}{A_k}
	\le \frac{4}{k}.
	\]
\end{proof}
\end{document}